\documentclass[11pt,reqno]{amsart}

\usepackage[T1]{fontenc}
\usepackage[utf8]{inputenc}
\usepackage{lmodern}
\usepackage{microtype}
\usepackage{amsmath,amssymb,amsthm,mathtools,mathrsfs}
\usepackage{booktabs,array}
\usepackage[hidelinks]{hyperref}
\hypersetup{
  pdftitle={Integral magneticity of the level-two K3 packet},
  pdfauthor={Alex Shvets},
  pdfsubject={Integral primitives of three level-two magnetic modular forms},
  pdfkeywords={magnetic modular form, CM theta lift, Weil representation, isogeny trace, hypergeometric function, K3 surface}
}

\allowdisplaybreaks
\numberwithin{equation}{section}
\newtheorem{theorem}{Theorem}[section]
\newtheorem{proposition}[theorem]{Proposition}
\newtheorem{lemma}[theorem]{Lemma}
\newtheorem{corollary}[theorem]{Corollary}
\theoremstyle{remark}

\newcommand{\Z}{\mathbf Z}

\newcommand{\C}{\mathbf C}

\newcommand{\D}{D}
\newcommand{\dd}{\,\mathrm d}
\newcommand{\ii}{\mathrm i}
\newcommand{\ord}{\operatorname{ord}}

\newcommand{\e}{\mathrm e}
\newcommand{\Tcal}{\mathcal T}
\newcommand{\Pcal}{\mathcal P}
\newcommand{\Fcal}{\mathcal F}
\newcommand{\Qcal}{\mathcal Q}

\title[Integral magneticity of the level-two K3 packet]
{Integral magneticity of the level-two K3 packet:\\
CM theta lifts and a $2$-isogeny trace contraction}

\author{Alex Shvets}
\address{Haifa, Israel}
\email{alex@shvets.io}
\urladdr{https://shvets.io}
\thanks{ORCID: 0009-0005-9802-379X}
\date{Version 1, July 20, 2026}


\begin{document}

\begin{abstract}
B\"onisch, Duhr, and Maggio introduced three meromorphic modular forms
$C_4,C_{6a},C_{6b}$ on $\Gamma_0(2)$, arising from a hypergeometric K3
family, and conjectured that they are magnetic of depths $1,2,2$.
Writing
\[
 C_4=\sum_{n\ge1}c_4(n)q^n,\qquad
 C_{6a}=\sum_{n\ge1}c_{6a}(n)q^n,\qquad
 C_{6b}=\sum_{n\ge1}c_{6b}(n)q^n,
\]
we prove the stronger denominator-one statements
\[
 \frac{c_4(n)}n,\qquad
 \frac{c_{6a}(n)}{n^2},\qquad
 \frac{c_{6b}(n)}{n^2}\in\Z
 \qquad(n\ge1).
\]
The weight-four case is reduced to a termwise binomial divisibility by a
hypergeometric change of Hauptmodul.  For weight six we identify the two
forms with canonical level-two CM forms of discriminants $-8$ and $-4$:
\[
 f_{3,-8,0,1,1}=-64C_{6a},\qquad
 f_{3,-4,2,1,1}=32C_{6b}.
\]
An explicit pair of vector-valued weakly holomorphic forms of weight
$-3/2$ then gives the full odd-prime divisibility through the
higher-level theta-lift coefficient formula of
L\"obrich--Schwagenscheidt.  The prime $2$ is treated independently.  If
$t=(\eta(2\tau)/\eta(\tau))^{24}$,
$H=\eta(\tau)^4/\eta(2\tau)^2$,
$J=2E_2(2\tau)-E_2(\tau)$, $u=64t$, and $\Tcal=H^4J$, we prove the
infinite-family contraction
\[
 U_2\bigl(\Tcal t\,\Z_2[[u]]\bigr)
 \subseteq 2^5\Tcal t\,\Z_2[[u]].
\]
Consequently $v_2(c_{6\bullet}(2^rm))\ge5r$, which is stronger than the
slope $2r$ required for double magneticity.  Thus the complete
level-two K3 packet is magnetic with global denominator one.
\end{abstract}

\maketitle
\tableofcontents

\section{Introduction and main theorem}

A meromorphic modular form
\[
 f(\tau)=\sum_{n\ne0}a(n)q^{n/N},
 \qquad q=\e^{2\pi\ii\tau},
\]
is called \emph{magnetic of depth $d$} when the algebraic numbers
$a(n)/n^d$ have globally bounded denominators.  The terminology was
introduced by Pa\c{s}ol and Zudilin \cite{PZ}, following the magnetic
example of Broadhurst--Zudilin and its proof by Li--Neururer
\cite{BZ,LN}.  The geometric framework of B\"onisch--Duhr--Maggio
\cite{BDM} produced a collection of new candidates associated with
families from mathematical physics.

The level-two K3 packet of \cite[Appendix~B.1]{BDM} is defined as follows.
Put
\begin{equation}\label{eq:def-t-J}
 t(\tau)=\left(\frac{\eta(2\tau)}{\eta(\tau)}\right)^{24},
 \qquad
 J(\tau)=2E_2(2\tau)-E_2(\tau).
\end{equation}
Then
\begin{align}
 C_4&=\frac{t}{3(1+64t)}\bigl(4J^2-E_4\bigr),
 \label{eq:def-C4}\\
 C_{6a}&=\frac{t(5+1408t+20480t^2)}{9(1-64t)^3}
          \bigl(8J^3+E_6\bigr),
 \label{eq:def-C6a}\\
 C_{6b}&=\frac{t(1-64t)}{9(1+64t)^2}
          \bigl(8J^3+E_6\bigr).
 \label{eq:def-C6b}
\end{align}
B\"onisch--Duhr--Maggio checked the required divisibilities for at least
the first $500$ coefficients but did not have a proof.  Allen--Long--Saad
later obtained Atkin--Swinnerton-Dyer congruences for meromorphic modular
forms and partial magneticity results, including control away from
specified small primes in relevant examples \cite{ALS}.

Our main result settles all three forms, with no residual denominator.

\begin{theorem}[Integral magneticity of the packet]\label{thm:main}
Let
\[
 C_4=\sum_{n\ge1}c_4(n)q^n,\qquad
 C_{6a}=\sum_{n\ge1}c_{6a}(n)q^n,\qquad
 C_{6b}=\sum_{n\ge1}c_{6b}(n)q^n.
\]
Then, for every $n\ge1$,
\begin{equation}\label{eq:main}
 n\mid c_4(n),\qquad
 n^2\mid c_{6a}(n),\qquad
 n^2\mid c_{6b}(n).
\end{equation}
Equivalently,
\[
 \D^{-1}C_4,\qquad \D^{-2}C_{6a},\qquad \D^{-2}C_{6b}
 \in q\Z[[q]],
 \qquad \D=q\frac{\dd}{\dd q}.
\]
\end{theorem}

The proof deliberately separates the odd and dyadic mechanisms in
weight six.  The odd-prime part comes from a CM theta lift and an explicit
negative-weight vector-valued form.  The dyadic part comes from a degree-two
modular equation and is stronger than required.  This separation also
makes the denominator-one conclusion robust: the intersection
$\Z[1/2]\cap\Z_2$ is $\Z$.

\section{The level-two modular dictionary and the three collapses}

Put
\begin{equation}\label{eq:tHJ}
 H=\frac{\eta(\tau)^4}{\eta(2\tau)^2},
 \qquad
 u=64t,
 \qquad
 \Tcal=H^4J.
\end{equation}
The product expansions begin
\begin{align*}
 t&=q+24q^2+300q^3+2624q^4+O(q^5),\\
 H&=1-4q+4q^2+O(q^3),\\
 J&=1+24q+24q^2+96q^3+O(q^4).
\end{align*}

\begin{proposition}[Level-two identities]\label{prop:dictionary}
One has
\begin{align}
 \D t&=tJ,\label{eq:Dt}\\
 J^2&=H^4(1+64t),\label{eq:J2}\\
 E_4&=H^4(1+256t),\label{eq:E4}\\
 E_6&=H^4J(1-512t),\label{eq:E6}\\
 H&={}_2F_1\!\left(\frac14,\frac14;1;-64t\right).
 \label{eq:Hhyper}
\end{align}
All branches in \eqref{eq:Hhyper} have constant term one at $q=0$.
\end{proposition}

\begin{proof}
The logarithmic derivative
$\D\log\eta(\delta\tau)=\delta E_2(\delta\tau)/24$ gives
\eqref{eq:Dt}.  The eta-quotient criterion shows that all expressions in
\eqref{eq:J2}--\eqref{eq:E6} are holomorphic modular forms on
$\Gamma_0(2)$ of the indicated weights.  Their first terms are
\begin{align*}
 J^2&=1+48q+O(q^2)=H^4(1+64t),\\
 E_4&=1+240q+O(q^2)=H^4(1+256t),\\
 E_6&=1-504q+O(q^2)=H^4J(1-512t).
\end{align*}
The Sturm bounds \cite{Sturm} at level $2$ in weights $4$ and $6$ are both $1$, so
these comparisons prove the three identities.

For \eqref{eq:Hhyper}, let
$\lambda=\vartheta_2^4/\vartheta_3^4$.  Standard theta products \cite{BB} give
\[
 t=\frac{\lambda^2}{256(1-\lambda)}.
\]
Kummer's quadratic transformation
\[
 {}_2F_1\!\left(\frac14,\frac14;1;
 -\frac{x^2}{4(1-x)}\right)
 =(1-x)^{1/4}{}_2F_1\!\left(\frac12,\frac12;1;x\right)
\]
together with
${}_2F_1(\frac12,\frac12;1;\lambda)=\vartheta_3^2$ and the standard
eta--theta identities gives \eqref{eq:Hhyper}.
\end{proof}

\begin{corollary}[Simultaneous collapse]\label{cor:collapse}
The three forms are
\begin{align}
 C_4&=\frac{tH^4}{1+u},\label{eq:C4collapse}\\
 C_{6a}&=\Tcal t\,\frac{5+22u+5u^2}{(1-u)^3},
 \label{eq:C6acollapse}\\
 C_{6b}&=\Tcal t\,\frac{1-u}{(1+u)^2}.
 \label{eq:C6bcollapse}
\end{align}
In particular, all three have integral Fourier coefficients.
\end{corollary}

\begin{proof}
Equations \eqref{eq:J2} and \eqref{eq:E4} give
$4J^2-E_4=3H^4$.  Moreover,
\[
 8J^3+E_6
 =8H^4J(1+64t)+H^4J(1-512t)=9H^4J.
\]
Substitution into \eqref{eq:def-C4}--\eqref{eq:def-C6b} gives the
formulas.  Their integrality follows from the eta products and the fact
that the three denominators have constant term one.
\end{proof}

We also record the Fricke data.  For
$W_2=\left(\begin{smallmatrix}0&-1\\2&0\end{smallmatrix}\right)$ use the
determinant-normalized slash action
\[
 (f|_kW_2)(\tau)=2^{k/2}(2\tau)^{-k}f\!\left(-\frac1{2\tau}\right).
\]
The eta transformation law gives
\begin{equation}\label{eq:fricke}
 t|_0W_2=2^{-12}t^{-1},\qquad
 H|_1W_2=-\ii2^{3/2}t^{1/4}H,\qquad
 J|_2W_2=-J.
\end{equation}
Consequently,
\begin{equation}\label{eq:uTfricke}
 u\!\left(-\frac1{2\tau}\right)=\frac1u,
 \qquad
 \Tcal|_6W_2=-u\Tcal.
\end{equation}

\section{The integral primitive of $C_4$}

This section gives a self-contained proof of the first assertion of
Theorem~\ref{thm:main}.  Define
\begin{equation}\label{eq:x}
 x=\frac{t}{1+64t}\in q\Z[[q]]
\end{equation}
and let $\Pcal$ be the unique series with $\Pcal(0)=0$ and
$\D\Pcal=C_4$.

\begin{proposition}[Hypergeometric primitive]\label{prop:C4primitive}
One has
\begin{equation}\label{eq:dPdx}
 \frac{\dd\Pcal}{\dd x}
 ={}_2F_1\!\left(\frac14,\frac34;1;64x\right)^2.
\end{equation}
If
\begin{equation}\label{eq:um}
 u_m=64^m\frac{(\frac14)_m(\frac34)_m}{(m!)^2}
 =\binom{4m}{2m}\binom{2m}{m},
\end{equation}
then
\begin{equation}\label{eq:Pexplicit}
 \Pcal=\sum_{N\ge1}\frac{x^N}{N}
 \sum_{r=0}^{N-1}u_ru_{N-1-r}.
\end{equation}
\end{proposition}

\begin{proof}
By \eqref{eq:C4collapse} and \eqref{eq:Dt},
\[
 \frac{\dd\Pcal}{\dd t}
 =\frac{H^4}{J(1+64t)}
 =\frac{H^2}{(1+64t)^{3/2}},
\]
where the branch is fixed by the constant term one.  Since
$\dd x/\dd t=(1+64t)^{-2}$,
\[
 \frac{\dd\Pcal}{\dd x}=H^2(1+64t)^{1/2}.
\]
Pfaff's transformation applied to \eqref{eq:Hhyper} gives
\[
 H=(1+64t)^{-1/4}
 {}_2F_1\!\left(\frac14,\frac34;1;
 \frac{64t}{1+64t}\right),
\]
which proves \eqref{eq:dPdx}.  Expansion and termwise integration give
\eqref{eq:um}--\eqref{eq:Pexplicit}.
\end{proof}

\begin{lemma}[Pairwise binomial divisibility]\label{lem:pairwise}
For all $r,s\ge0$,
\begin{equation}\label{eq:pairwise}
 r+s+1\mid u_ru_s.
\end{equation}
\end{lemma}

\begin{proof}
Fix a prime $p$ and put $\alpha=v_p(r+s+1)$.  For $q=p^\nu$,
Legendre's formula gives the $q$-level contribution
\[
 \delta_q(n)=
 \left\lfloor\frac{4n}{q}\right\rfloor
 -\left\lfloor\frac{2n}{q}\right\rfloor
 -2\left\lfloor\frac{n}{q}\right\rfloor
\]
to $v_p(u_n)$.  If $n=aq+\rho$, $0\le\rho<q$, then
\[
 \delta_q(n)=
 \left\lfloor\frac{4\rho}{q}\right\rfloor
 -\left\lfloor\frac{2\rho}{q}\right\rfloor.
\]
Thus $\delta_q(n)=0$ exactly when $4\rho<q$, and otherwise
$\delta_q(n)\ge1$.

For $1\le\nu\le\alpha$, let $\rho_r,\rho_s$ be the residues of $r,s$
modulo $q=p^\nu$.  Since $r+s\equiv-1\pmod q$,
$\rho_r+\rho_s=q-1$.  The inequalities $4\rho_r<q$ and
$4\rho_s<q$ cannot both hold.  Hence
$\delta_{p^\nu}(r)+\delta_{p^\nu}(s)\ge1$ for every
$1\le\nu\le\alpha$.  Summing the nonnegative Legendre contributions
proves
$v_p(u_ru_s)\ge\alpha=v_p(r+s+1)$.
\end{proof}

\begin{corollary}\label{cor:C4done}
For every $n\ge1$, $n\mid c_4(n)$.
\end{corollary}

\begin{proof}
In \eqref{eq:Pexplicit}, take $s=N-1-r$ in
Lemma~\ref{lem:pairwise}.  Every summand is divisible by $N$, so
$\Pcal\in x\Z[[x]]$.  Since $x\in q\Z[[q]]$, one has
$\Pcal\in q\Z[[q]]$.  Comparing coefficients in
$\D\Pcal=C_4$ proves the claim.
\end{proof}

\section{The vector-valued negative-weight input}

We now construct the forms that control the odd-prime part of the
weight-six candidates.  Let
\[
 A=\Z/4\Z,\qquad Q(r)=\frac{r^2}{8}\pmod\Z,
\]
and let $\rho_2$ be the Weil representation of the finite quadratic
module $(A,Q)$.  Write $\{\mathfrak e_r:r\in A\}$ for the standard basis
of $\C[A]$.  The unary theta vector
\begin{equation}\label{eq:theta-vector}
 \boldsymbol\theta(\tau)
 =\sum_{r\bmod4}\theta_r(\tau)\mathfrak e_r,
 \qquad
 \theta_r(\tau)=\sum_{n\equiv r\ (4)}q^{n^2/8},
\end{equation}
transforms with weight $1/2$ for $\rho_2$; see, for example, the standard Weil-representation construction in \cite{Bruinier}.

For a modular form of weight $\kappa$, put
\begin{equation}\label{eq:serre}
 \mathscr D_\kappa=\D-\frac{\kappa}{12}E_2.
\end{equation}
Set
\begin{align}
 A_0&=E_{10}\boldsymbol\theta,\nonumber\\
 A_1&=E_8\mathscr D_{1/2}\boldsymbol\theta,\nonumber\\
 A_2&=E_6\mathscr D_{5/2}\mathscr D_{1/2}
      \boldsymbol\theta.
 \label{eq:A012}
\end{align}
All three have weight $21/2$ for $\rho_2$.

\begin{proposition}[Explicit Poincar\'e series]\label{prop:Pexplicit}
The vector-valued forms
\begin{align}
 \Pcal_0&=\Delta^{-1}
 \left(\frac{11}{9}A_0-12A_1+32A_2\right),
 \label{eq:P0}\\
 \Pcal_2&=\Delta^{-1}
 \left(-\frac1{54}A_0+\frac23A_1+\frac{16}{3}A_2\right)
 \label{eq:P2}
\end{align}
are weakly holomorphic modular forms of weight $-3/2$ for $\rho_2$.
They are the Maass--Poincar\'e series
\begin{equation}\label{eq:P-identification}
 \Pcal_0=P_{-3/2,8,0},\qquad
 \Pcal_2=P_{-3/2,4,2},
\end{equation}
normalized by the principal parts
\begin{equation}\label{eq:principal-parts}
 \Pcal_0=2q^{-1}\mathfrak e_0+O(1),\qquad
 \Pcal_2=2q^{-1/2}\mathfrak e_2+O(1).
\end{equation}
Moreover, every Fourier coefficient of $\Pcal_0$ and $\Pcal_2$ lies in
$\Z[1/2]$.
\end{proposition}

\begin{proof}
The Serre derivative preserves the representation, so
\eqref{eq:P0}--\eqref{eq:P2} have the asserted transformation law and
weight.  The leading coefficients of
$\boldsymbol\theta$, $\mathscr D_{1/2}\boldsymbol\theta$, and
$\mathscr D_{5/2}\mathscr D_{1/2}\boldsymbol\theta$ in the relevant
components are
\begin{center}
\begin{tabular}{c c c c}
\toprule
component & $\boldsymbol\theta$ &
$\mathscr D_{1/2}\boldsymbol\theta$ &
$\mathscr D_{5/2}\mathscr D_{1/2}\boldsymbol\theta$\\
\midrule
$r=0$, exponent $0$ & $1$ & $-1/24$ & $5/576$\\
$r=1,3$, exponent $1/8$ & $1$ & $1/12$ & $-1/144$\\
$r=2$, exponent $1/2$ & $2$ & $11/12$ & $77/288$\\
\bottomrule
\end{tabular}
\end{center}
Substitution shows that the numerator of \eqref{eq:P0} has constant
term $2\mathfrak e_0$ and no other term below $q$, while the numerator
of \eqref{eq:P2} has leading term $2q^{1/2}\mathfrak e_2$ and no other
term below $q$.  Since $\Delta=q+O(q^2)$, this proves
\eqref{eq:principal-parts}.

We next show that the relevant harmonic Maass--Poincar\'e series are
weakly holomorphic.  Their shadows lie in
$S_{7/2}(\rho_2^*)$.  Theta decomposition identifies this space with
the Jacobi cusp space $J_{4,2}^{\mathrm{cusp}}$.  By the
Eichler--Zagier structure theorem \cite{EZ}, every weak Jacobi form of weight $4$
and index $2$ is a linear combination of
\[
 E_4\phi_{0,1}^2,\qquad
 E_6\phi_{-2,1}\phi_{0,1},\qquad
 E_4^2\phi_{-2,1}^2.
\]
With $X=\zeta+\zeta^{-1}$, their $q^0$ terms are respectively
\[
 (X+10)^2,\qquad (X-2)(X+10),\qquad (X-2)^2.
\]
The determinant of their coefficient vectors in $1,X,X^2$ is $-1728$.
A Jacobi cusp form has zero $q^0$ term, so all three coefficients must
vanish.  Hence
\begin{equation}\label{eq:shadow-zero}
 S_{7/2}(\rho_2^*)=0.
\end{equation}
Thus the two harmonic Maass--Poincar\'e series in
\eqref{eq:P-identification} are weakly holomorphic.  Their differences
from \eqref{eq:P0} and \eqref{eq:P2} are holomorphic forms of negative
weight and therefore vanish.

It remains to control denominators.  Work componentwise in
\[
 R=\Z[1/2][[q^{1/8}]].
\]
Using Ramanujan's identity for $\D E_2$,
\begin{equation}\label{eq:second-serre}
 \mathscr D_{5/2}\mathscr D_{1/2}\boldsymbol\theta
 =\D^2\boldsymbol\theta-\frac14E_2\D\boldsymbol\theta
 +\frac{3E_2^2+2E_4}{576}\boldsymbol\theta.
\end{equation}
Let $N_0,N_2$ be the numerators in \eqref{eq:P0}, \eqref{eq:P2}.
A direct collection of terms gives
\begin{align}
 N_0={}&32E_6\D^2\boldsymbol\theta
 -(12E_4^2+8E_6E_2)\D\boldsymbol\theta\nonumber\\
 &+\left(\frac43E_4E_6+\frac12E_4^2E_2
 +\frac16E_6E_2^2\right)\boldsymbol\theta.
 \label{eq:N0}
\end{align}
The only possible odd denominator in $N_0$ occurs in the last bracket,
which equals
\[
 \frac16\bigl(8E_4E_6+3E_4^2E_2+E_6E_2^2\bigr).
\]
Its numerator is divisible by $3$ in $R$, since
\[
 8E_4E_6+3E_4^2E_2+E_6E_2^2\equiv8+3+1\equiv0\pmod3
\]
and $E_2\equiv E_4\equiv E_6\equiv1\pmod3$.  Hence $N_0\in R$.
Similarly,
\begin{align}
 N_2=\frac13\biggl(&16E_6\D^2\boldsymbol\theta
 +(2E_4^2-4E_6E_2)\D\boldsymbol\theta\nonumber\\
 &+\frac{E_2(E_6E_2-E_4^2)}{12}\boldsymbol\theta\biggr).
 \label{eq:N2}
\end{align}
Since $E_6E_2-E_4^2=2\D E_6$, the last term inside the parentheses is
divisible by $3$ in $R$.  Modulo $3$, the first two terms reduce to
$(\D^2+\D)\boldsymbol\theta$.  On a theta monomial $q^{n^2/8}$ its
coefficient is
\[
 \frac{n^2}{8}\left(\frac{n^2}{8}+1\right)\equiv0\pmod3:
\]
if $3\mid n$ the first factor is $0$, while if $3\nmid n$ then $n^2/8\equiv-1\pmod3$ and the second factor is $0$.  Thus $N_2\in R$.
Finally, $\Delta^{-1}\in q^{-1}\Z[[q]]$, proving the asserted
$2$-integrality.
\end{proof}

We will need one exact positive coefficient from each form.

\begin{lemma}\label{lem:first-P-coeff}
Writing
\[
 \Pcal_j=\sum_{r\bmod4}\sum_{n\gg-\infty}
 a_j(n,r)q^{n/8}\mathfrak e_r,
 \qquad j\in\{0,2\},
\]
one has
\begin{equation}\label{eq:first-P-coeff}
 a_0(1,1)=-640,
 \qquad
 a_2(1,1)=64.
\end{equation}
\end{lemma}

\begin{proof}
In component $r=1$, the coefficients of $(A_0,A_1,A_2)$ at exponents
$1/8$ and $9/8$ are
\[
 \begin{array}{c|ccc}
 &A_0&A_1&A_2\\ \hline
 q^{1/8}&1&1/12&-1/144\\
 q^{9/8}&-263&505/12&839/144.
 \end{array}
\]
The two linear combinations in \eqref{eq:P0}--\eqref{eq:P2} vanish in
degree $1/8$ and equal $-640$, respectively $64$, in degree $9/8$.
Since $\Delta^{-1}=q^{-1}+24+O(q)$, this is exactly
\eqref{eq:first-P-coeff}.
\end{proof}

\section{Canonical level-two CM forms}

For $d\in\{-8,-4\}$ and $r\in\{0,2\}$ with $d\equiv r^2\pmod8$,
let
\[
 \Qcal_{d,r}^{(2)}=
 \left\{[a,b,c]:
 \begin{array}{l}
 a,b,c\in\Z,\ a>0,\ b^2-4ac=d,\\
 2\mid a,\ b\equiv r\pmod4
 \end{array}\right\}.
\]
Define the weight-six canonical CM form
\begin{equation}\label{eq:Fdr}
 F_{d,r}(\tau)=\frac{|d|^{5/2}}{4\pi^3}
 \sum_{Q=[a,b,c]\in\Qcal_{d,r}^{(2)}}Q(\tau,1)^{-3}.
\end{equation}
This is the specialization $f_{3,d,r,1,1}$ of the higher-level forms of
L\"obrich--Schwagenscheidt \cite[Section~6]{LS}.  It is a meromorphic
modular form of weight $6$ for $\Gamma_0(2)$, vanishing at the cusps,
with poles at the associated CM points.

\begin{lemma}[The two CM orbits]\label{lem:CM-orbits}
The set $\Qcal_{-8,0}^{(2)}$ is a single $\Gamma_0(2)$-orbit represented
by
\[
 Q_8=[2,0,1],\qquad \alpha_8=\frac{\ii}{\sqrt2}.
\]
The set $\Qcal_{-4,2}^{(2)}$ is a single $\Gamma_0(2)$-orbit represented
by
\[
 Q_4=[2,2,1],\qquad \alpha_4=\frac{-1+\ii}{2}.
\]
\end{lemma}

\begin{proof}
Both discriminants have class number one.  If
$Q_8|\gamma\in\Qcal_{-8,0}^{(2)}$ for
$\gamma=\left(\begin{smallmatrix}a&b\\c&d\end{smallmatrix}\right)$,
then its leading coefficient $2a^2+c^2$ is even, hence $c$ is even.
Thus $\gamma\in\Gamma_0(2)$; its middle coefficient is automatically
divisible by $4$.

For $Q_4$, the leading coefficient is $2a^2+2ac+c^2$.  Its evenness
again forces $c$ even, and then $a,d$ are odd.  The transformed middle
coefficient is
\[
 2(2ab+ad+bc+cd)\equiv2\pmod4.
\]
Thus the full class-number-one orbit does not split further under the
stated level conditions.
\end{proof}

\begin{lemma}[Fricke symmetry]\label{lem:F-fricke}
For $F=F_{-8,0}$ or $F_{-4,2}$,
\begin{equation}\label{eq:F-fricke}
 F|_6W_2=F.
\end{equation}
\end{lemma}

\begin{proof}
The involution
\[
 [a,b,c]\longmapsto[2c,-b,a/2]
\]
permutes either indexing set.  Moreover,
\[
 Q\!\left(-\frac1{2\tau},1\right)
 =\frac1{2\tau^2}[2c,-b,a/2](\tau,1).
\]
The factor $(2\tau^2)^3$ is exactly canceled by the normalized
weight-six slash factor, proving \eqref{eq:F-fricke}.
\end{proof}

We now identify the two candidates with these canonical forms.

\begin{theorem}[CM identifications]\label{thm:CM-identification}
One has the exact identities
\begin{equation}\label{eq:CM-identification}
 F_{-8,0}=-64C_{6a},
 \qquad
 F_{-4,2}=32C_{6b}.
\end{equation}
\end{theorem}

\begin{proof}
Write $R=F/\Tcal$.  Since $u$ is a Hauptmodul for $X_0(2)$,
$R$ is a rational function of $u$.  From
\eqref{eq:uTfricke} and Lemma~\ref{lem:F-fricke},
\begin{equation}\label{eq:R-functional}
 R(1/u)=-\frac1uR(u).
\end{equation}
The forms $F$ vanish at both cusps, while $\Tcal$ has order $0$ at
$\infty$ and order $1$ at $0$.  Hence $R(0)=0$ and $R$ is finite at
$u=\infty$.

The point $\alpha_8$ is fixed by $W_2$.  Equation
$t|W_2=2^{-12}t^{-1}$ and positivity on the imaginary axis give
$u(\alpha_8)=1$.  The form $F_{-8,0}$ has a triple pole there and no
other pole modulo $\Gamma_0(2)$.  At the order-two elliptic point
$\alpha_4$, the weight-two form $J$ vanishes; hence \eqref{eq:J2} gives
$u(\alpha_4)=-1$.  If $m=\ord_{\alpha_4}(1+u)$, then
$\D u=uJ$ gives $\ord_{\alpha_4}(J)=m-1$, while \eqref{eq:J2}
gives $2(m-1)=m$.  Thus $m=2$ and $\Tcal=H^4J$ has a simple zero.
The form $F_{-8,0}$ is holomorphic at $\alpha_4$, and every local
holomorphic weight-six form vanishes there because the elliptic
stabilizer has automorphy factor $\ii$ and $\ii^6=-1$.  Therefore
$F_{-8,0}/\Tcal$ is holomorphic.  Therefore
\[
 R_8(u)=\frac{A(u)}{(1-u)^3},\qquad \deg A\le3,
 \qquad A(0)=0.
\]
The functional equation \eqref{eq:R-functional} is equivalent to
$A(u)=u^4A(1/u)$, hence
\begin{equation}\label{eq:R8-shape}
 R_8(u)=\frac{u(a+bu+au^2)}{(1-u)^3}.
\end{equation}

For $F_{-4,2}$, the only pole is the triple pole at $\alpha_4$.  Since
$1+u$ has a double zero and $\Tcal$ a simple zero there,
$F_{-4,2}/\Tcal$ has a pole of order at most two in the $u$-coordinate.
Thus
\[
 R_4(u)=\frac{A(u)}{(1+u)^2},\qquad \deg A\le2,
 \qquad A(0)=0.
\]
Equation \eqref{eq:R-functional} gives $A(u)=-u^3A(1/u)$, so
\begin{equation}\label{eq:R4-shape}
 R_4(u)=a\frac{u(1-u)}{(1+u)^2}.
\end{equation}

It remains to determine the constants.  The higher-level Fourier formula
of \cite[Proposition~6.1]{LS}, specialized to
$N=2$, $k=3$, $D=\rho=1$, gives at $n=1$
\begin{equation}\label{eq:F-first-from-P}
 2[q]F_{-8,0}=a_0(1,1),
 \qquad
 2[q]F_{-4,2}=a_2(1,1).
\end{equation}
By Lemma~\ref{lem:first-P-coeff},
\begin{equation}\label{eq:F-first}
 [q]F_{-8,0}=-320,
 \qquad
 [q]F_{-4,2}=32.
\end{equation}
Since $u=64q+O(q^2)$ and $\Tcal=1+O(q)$,
\eqref{eq:R8-shape} gives $a=-5$, while
\eqref{eq:R4-shape} gives $a=1/2$.

For $R_8$ one further constant remains.  Near
$\alpha_8=\ii/\sqrt2$, only $Q_8=2\tau^2+1$ contributes to the leading
pole.  The normalization in \eqref{eq:Fdr} gives
\begin{equation}\label{eq:F8-local}
 F_{-8,0}(\tau)
 \sim\frac{32\sqrt2}{\pi^3}(2\tau^2+1)^{-3}
 \sim\frac{2\ii}{\pi^3}(\tau-\alpha_8)^{-3}.
\end{equation}
On the other hand, $\D u=uJ$ and
$\Tcal=J^3/(1+u)$.  Formula \eqref{eq:R8-shape} therefore has leading
coefficient
\begin{equation}\label{eq:R8-local}
 -\frac{\ii}{16\pi^3}(2a+b)(\tau-\alpha_8)^{-3}.
\end{equation}
Comparison with \eqref{eq:F8-local} gives $2a+b=-32$, hence $b=-22$.
Thus
\[
 F_{-8,0}=-\Tcal\frac{u(5+22u+5u^2)}{(1-u)^3}
 =-64C_{6a},
\]
and
\[
 F_{-4,2}=\frac12\Tcal\frac{u(1-u)}{(1+u)^2}
 =32C_{6b}.
\]
\end{proof}

\section{The full odd-prime divisibility}

We now use the complete coefficient formula, not merely its first term.
For $j\in\{0,2\}$ retain the notation
\[
 \Pcal_j=\sum_{r\bmod4}\sum_n a_j(n,r)q^{n/8}\mathfrak e_r.
\]

\begin{proposition}[Specialized theta-lift coefficient formula]
\label{prop:lift-formula}
For every $n\ge1$,
\begin{align}
 2c_{F_{-8,0}}(n)
 &=n^2\sum_{s\mid n}s^3a_0(s^2,s),
 \label{eq:F8-coeff}\\
 2c_{F_{-4,2}}(n)
 &=n^2\sum_{s\mid n}s^3a_2(s^2,s),
 \label{eq:F4-coeff}
\end{align}
where the second argument is read modulo $4$.
\end{proposition}

\begin{proof}
Proposition~6.1 of \cite{LS} states in the present case that
\[
 2c_F(n)=n^5\sum_{m\mid n}m^{-3}
 a\!\left(\frac{n^2}{m^2},\frac nm\right).
\]
Putting $s=n/m$ gives \eqref{eq:F8-coeff} and
\eqref{eq:F4-coeff}.  The factor $2$ on the left occurs because
$r=-r$ in $\Z/4\Z$ for $r=0,2$.
\end{proof}

\begin{corollary}[Odd-prime magneticity]\label{cor:odd}
For every $n\ge1$,
\begin{equation}\label{eq:odd}
 \frac{c_{6a}(n)}{n^2},\qquad
 \frac{c_{6b}(n)}{n^2}\in\Z[1/2].
\end{equation}
\end{corollary}

\begin{proof}
All coefficients $a_j(s^2,s)$ lie in $\Z[1/2]$ by
Proposition~\ref{prop:Pexplicit}.  Equations
\eqref{eq:F8-coeff}--\eqref{eq:F4-coeff} therefore show that
$c_F(n)/n^2\in\Z[1/2]$.  The powers of $2$ in the identifications
\eqref{eq:CM-identification} do not change this conclusion.
\end{proof}

In particular, the prime $3$ is already included in
Corollary~\ref{cor:odd}; no separate experimental $3$-adic recurrence is
needed.

\section{A $2$-isogeny trace contraction}

The remaining prime is treated by an infinite-family statement.  For
$f(z)\in\Z_2[[z]]$ define
\begin{equation}\label{eq:Ff}
 \Fcal_f(\tau)=\Tcal(\tau)t(\tau)f(u(\tau)),
 \qquad u=64t.
\end{equation}

\begin{theorem}[$U_2$ contraction]\label{thm:U2}
One has the module inclusion
\begin{equation}\label{eq:U2-contraction}
 U_2\bigl(\Tcal t\,\Z_2[[u]]\bigr)
 \subseteq2^5\Tcal t\,\Z_2[[u]].
\end{equation}
Consequently, if
$\Fcal_f=\sum_{n\ge1}a_f(n)q^n$, then
\begin{equation}\label{eq:U2-valuation}
 v_2\bigl(a_f(2^rm)\bigr)\ge5r
 \qquad(m,r\ge1).
\end{equation}
\end{theorem}

\begin{proof}
Put
\[
 s_0=t\!\left(\frac\tau2\right),\qquad
 s_1=t\!\left(\frac{\tau+1}{2}\right),\qquad
 y_i=8s_i.
\]
The degree-two modular equation is
\begin{equation}\label{eq:modular-equation}
 s_0+s_1=48t+4096t^2,
 \qquad
 s_0s_1=-t.
\end{equation}
For the product identity, set $x=q^{1/2}$ and use
$t(q)=q\prod_{n\ge1}(1+q^n)^{24}$.  Then
$t(x)t(-x)=-t(x^2)$; the reduction is Euler's identity
\[
 \prod_{m\ge1}(1+y^m)
 =\prod_{m\ge1}(1-y^{2m-1})^{-1}.
\]
For the sum, $s_0+s_1=2U_2t$ is the trace along the degree-two
correspondence and hence a modular function on $X_0(2)$.  It vanishes at
$\infty$ and has a pole of order at most two at the other cusp, so it is
$At+Bt^2$.  The coefficients of $q$ and $q^2$ give $A=48$, $B=4096$.
Equivalently,
\begin{equation}\label{eq:y-equation}
 y_0+y_1=6u+8u^2,
 \qquad
 y_0y_1=-u.
\end{equation}

Let $J_i=J((\tau+i)/2)$.  If $y=y_i$ is one root of
\[
 y^2-(6u+8u^2)y-u=0,
\]
implicit differentiation, using
$\D u=uJ$ and $\D y_i=y_iJ_i/2$, gives
\begin{equation}\label{eq:Jbranch}
 \frac{J_i}{J}
 =\frac{2u((6+16u)y_i+1)}
 {y_i(2y_i-6u-8u^2)}.
\end{equation}
A direct reduction of the cube of \eqref{eq:Jbranch} by the quadratic
relation yields
\begin{equation}\label{eq:Jcube}
 \left(\frac{J_i}{J}\right)^3
 =\frac{(1-y_{1-i})(1+8y_i)}{(1+u)^2}.
\end{equation}
Since $\Tcal=J^3/(1+u)$, it follows that
\begin{equation}\label{eq:Tbranch}
 \frac{\Tcal((\tau+i)/2)}{\Tcal(\tau)}
 =\frac{1-y_{1-i}}{1+u}.
\end{equation}

Using
$U_2F=\frac12(F(\tau/2)+F((\tau+1)/2))$ and
$y_i(1-y_{1-i})=y_i+u$, we obtain
\begin{equation}\label{eq:U2-master}
 \frac{U_2\Fcal_f}{\Tcal}
 =\frac1{16(1+u)}
 \sum_{i=0}^1(y_i+u)f(8y_i).
\end{equation}
Write $f(z)=\sum_{k\ge0}a_kz^k$ and
$P_m=y_0^m+y_1^m$.  Equation \eqref{eq:y-equation} gives
\begin{equation}\label{eq:Pm}
 P_0=2,\qquad P_1=6u+8u^2,
 \qquad
 P_m=(6u+8u^2)P_{m-1}+uP_{m-2}.
\end{equation}
Hence $P_m\in u\Z[u]$ for every $m\ge1$.  Expanding
\eqref{eq:U2-master},
\begin{equation}\label{eq:U2-expanded}
 \frac{U_2\Fcal_f}{\Tcal}
 =\frac1{16(1+u)}\sum_{k\ge0}
 a_k8^k(P_{k+1}+uP_k).
\end{equation}
For $k=0$,
$P_1+uP_0=8u(1+u)$.  For $k\ge1$, the numerator
$8^k(P_{k+1}+uP_k)$ belongs to $8u\Z_2[u]$.  Therefore the right-hand
side of \eqref{eq:U2-expanded} belongs to
\[
 \frac u2\Z_2[[u]]=32t\Z_2[[u]],
\]
which proves \eqref{eq:U2-contraction}.  Iteration gives
$U_2^r\Fcal_f\in2^{5r}\Z_2[[q]]$.  Since
$U_2^r\Fcal_f=\sum_na_f(2^rn)q^n$, this is
\eqref{eq:U2-valuation}.
\end{proof}

\begin{corollary}[Dyadic slopes for the two candidates]\label{cor:dyadic}
For all $m,r\ge1$,
\begin{equation}\label{eq:dyadic}
 v_2(c_{6a}(2^rm))\ge5r,
 \qquad
 v_2(c_{6b}(2^rm))\ge5r.
\end{equation}
\end{corollary}

\begin{proof}
The two rational series
\[
 \frac{5+22u+5u^2}{(1-u)^3},
 \qquad
 \frac{1-u}{(1+u)^2}
\]
belong to $\Z_2[[u]]$.  Apply Theorem~\ref{thm:U2} to
\eqref{eq:C6acollapse} and \eqref{eq:C6bcollapse}.
\end{proof}

\section{Completion of the proof and initial data}

\begin{proof}[Proof of Theorem~\ref{thm:main}]
The $C_4$ assertion is Corollary~\ref{cor:C4done}.  For either weight-six
form, Corollary~\ref{cor:odd} gives
$c(n)/n^2\in\Z[1/2]$.  Write $n=2^rm$ with $m$ odd.  When $r=0$, the quotient is already
$2$-adically integral because $c(n)\in\Z$ and $n$ is odd.  For $r\ge1$,
Corollary~\ref{cor:dyadic} gives
\[
 v_2\!\left(\frac{c(n)}{n^2}\right)
 \ge5r-2r=3r\ge0.
\]
Thus $c(n)/n^2\in\Z_2$.  Since
$\Z[1/2]\cap\Z_2=\Z$, one has $c(n)/n^2\in\Z$.
\end{proof}

For normalization, the three forms begin
\begin{align*}
 C_4={}&q-56q^2+2076q^3-65984q^4+1941630q^5+O(q^6),\\
 C_{6a}={}&5q+2528q^2+547524q^3+87849984q^4
 +12091540750q^5+O(q^6),\\
 C_{6b}={}&q-160q^2+9972q^3-447488q^4
 +17028150q^5+O(q^6).
\end{align*}
Their integral primitives are
\begin{align*}
 \D^{-1}C_4={}&q-28q^2+692q^3-16496q^4+388326q^5+O(q^6),\\
 \D^{-2}C_{6a}={}&5q+632q^2+60836q^3+5490624q^4
 +483661630q^5+O(q^6),\\
 \D^{-2}C_{6b}={}&q-40q^2+1108q^3-27968q^4
 +681126q^5+O(q^6).
\end{align*}
These expansions are checks only.  No finite verification is used as a
proof input.

\section{Scope and reproducibility}

The theorem closes the three $\Gamma_0(2)$ candidates in
\cite[Appendix~B.1]{BDM} and strengthens magneticity from bounded
denominators to denominator one.  Two features extend beyond the three
examples:
\begin{enumerate}
\item the explicit vector-valued construction isolates the entire
odd-prime obstruction in a $2$-integral negative-weight input;
\item the inclusion \eqref{eq:U2-contraction} applies to every
$f\in\Z_2[[u]]$, not merely to the two rational functions appearing in
$C_{6a}$ and $C_{6b}$.
\end{enumerate}
Thus the proof supplies a general bad-prime mechanism rather than a
coefficientwise verification of two isolated forms.

The ancillary program
\texttt{verify\_level\_two\_k3\_packet.py} uses exact integer and rational
arithmetic.  It verifies the modular identities through $q^{120}$, the
three divisibilities through $n=120$, the vector-valued principal parts
and $2$-integrality, the specialized CM coefficient identities through
$n=50$, the degree-two modular equation, and the algebraic branch
identity underlying \eqref{eq:Jcube}.  Its SHA-256 digest is
\begin{center}
\texttt{423d1de87ba9dd04788c4418617b66058c0ed84eaa789ad3d2b85ddf35fb3d00}.
\end{center}
The finite checks are proof-independent audits.  The only non-elementary
external coefficient identity used in the proof is the published
higher-level theta-lift formula \cite[Proposition~6.1]{LS}; all of its
parameters and normalizations are specialized explicitly in
Proposition~\ref{prop:lift-formula}.

\end{document}